\documentclass[a4paper,11pt]{article}

\usepackage{fullpage}
\usepackage{multicol, float}
\usepackage{lineno,hyperref}
\usepackage{verbatim, euscript}
\usepackage{subfigure}
\usepackage{bbm,amssymb}
\usepackage{amsfonts}
\usepackage{amsmath, mathtools}
\usepackage{pictexwd}
\usepackage{amsthm}
\usepackage{mathrsfs}
\usepackage{mathtools}
\mathtoolsset{showonlyrefs}

\theoremstyle{plain}

\newtheorem{thm}{Theorem}
\newtheorem{lem}{Lemma}[section]

\newtheorem{rem}[lem]{Remark}

\newtheorem{prop}[lem]{Proposition}

\newcommand{\eps}{\varepsilon}
\newcommand{\E}{\mathbb{E}}

\newcommand{\1}{\mathbbm{1}} 
\newcommand{\N}{\mathbb{N}}

\newcommand{\R}{\mathbb{R}}

\newcommand{\V}{\text{Var}}
\newcommand{\EE}[1]{\mathbb{E} \left[ #1 \right]}
\newcommand{\VV}[1]{\V \left( #1 \right)}
\newcommand{\pp}[1]{\mathbb{P} \left( #1 \right)}

\newcommand{\e}[1]{\mathbb{E}\left[#1\right]}

\begin{document}

\title{Haldane's asymptotics for supercritical branching processes in an iid random environment}

\author{Florin Boenkost\thanks{Institut f\"ur Mathematik, Universität Wien, Wien, Austria, florin.boenkost@univie.ac.at} , Götz  Kersting\thanks{Institut f\"ur Mathematik, Goethe Universit\"at, Frankfurt am Main, Germany, kersting@math.uni-frankfurt.de}}
\maketitle





\begin{abstract}
	Branching processes in a random environment are natural generalisations of Galton-Watson processes. In this paper we analyse the asymptotic decay of the survival probability for a sequence of slightly supercritical branching processes in an iid random environment, where the offspring expectation converges from above to $1$. We prove that Haldane's asymptotics, known from classical Galton-Watson processes, turns up again in the random environment case, provided that one stays away from the critical/subcritical regime. A central building block is a connection to and a limit theorem for perpetuities with asymptotically vanishing interest rates. \\\\
	
\end{abstract}

\section{Introduction and main result} \label{Sec Introduction}
In the early twentieth century Fisher \cite{Fisher1923}, Haldane \cite{Haldane1927} and Wright \cite{Wright1931} studied the survival probability of a beneficial mutant gene in large populations. They argued that, as long as the mutant is sufficiently rare and the selective advantage is small, the number of mutants should evolve like a slightly supercritical Galton-Watson process (GWP). As Haldane concluded, the probability $\pi$ of ultimate survival should obey the asymptotics
\begin{align}
	\pi \approx \frac{2\eps}{\sigma^2} \label{Survival prob asymp}
\end{align}
for small $\eps>0$, where $1+\eps$ is the offspring expectation and $\sigma^2$ the offspring variance. This approximation gained a lot of attention in the literature, see \cite{Patwa2008} for an overview. For GWPs this asymptotics was considered among others by Kolmogorov \cite{Kolmogorov1938}, Eshel \cite{Eshel1981}, Athreya \cite{Athreya1992} and Hoppe \cite{Hoppe1992}.

In this paper we consider the asymptotic survival probability of a slightly supercritical branching process in a random environment. As it turns out Haldane's asymptotics 
remains valid for
such processes in case of an iid random environment, as long as one keeps away from the domain of subcritical behaviour. This might come as a surprise, since the accompanying Kolmogorov asymptotics for the survival probability of critical GWPs fails in the case of an iid random environment \cite{Geiger2000}. Close to subcriticality one observes a smooth adaption of Haldane's formula.

Let us recall the notion of a branching process in random environment (BPRE). Denote by $\mathcal{P}(\N_0)$ the space of all probability measures on $\N_0=\{0,1,2,...\}$. 
Endow  $\mathcal{P}(\N_0)$ with the total variation metric and the induced Borel-$\sigma$-algebra. This allows to consider random probability measures on $\mathcal{P}(\N_0)$, namely random variables $P$ with values in $\mathcal{P}(\N_0)$. Relevant quantities related to a random measure $P$ are  its (random) mean and (random) second factorial moment,
\begin{align}
	M=\sum_{z=1}^{\infty} z P[z], \qquad M^{(2)}=\sum_{z=2}^{\infty} z(z-1) P[z],
\end{align} 
where $P[z]$ denotes the (random) weight of $P$ at $z\in \mathbb N_0$. 
Next,  a random environment $\mathsf V=(P_1,P_2,...)$ is a sequence of random probability measures. 
Given such an environment  we call  $\mathsf Z=(Z_k,k\geq 0)$ a branching process in the  random environment $\mathsf V$, if it has the representation
\begin{align}
	Z_k= \sum_{i=1}^{Z_{k-1}} U_{i,k}, \qquad Z_0=1,
\end{align}
where conditionally on $\mathsf V$ the family $(U_{i,k}, i \geq 1, k \geq 1)$ is independent and for each $k\in \N$ the sequence $(U_{i,k}, i \geq 1)$ is identically distributed with distribution $P_k$. 
Here, we consider BPREs in an iid random environment, in this case the environment $\mathsf V=(P_1,P_2,...)$ consists of independent copies of a generic random measure $\mathsf P$. 
Then, the random variables $U_{i,k}$ are (unconditionally) identically distributed, thus copies of a generic random variable $\mathsf U$. For an account on BPREs we refer to \cite{Kersting2017Book} and the literature cited therein.

A BPRE in an iid environment is called subcritical, critical or supercritical, if $\mathbb E[\log \mathsf M] $ is less than, equal to or bigger than 0 (provided the existence of the expectation). If $\mathbb E[\log \mathsf M] \le 0$, then the probability of ultimate survival
\begin{align*}
	\pi : = \lim_{k\to \infty} \mathbb P(Z_k>0)
\end{align*}
vanishes \cite[Theorem 2.1]{Kersting2017Book}. In particular $\mathbb E[\mathsf M] \le 1$ implies a.s. ultimate extinction, as follows from Jensen's inequality. Thus, we may observe a  positive probability $\pi$ of survival only in the case of $\mathbb E[\mathsf M]>1$.

In general, BPREs and GWPs differ a lot in their properties. However, this concerns mainly the subcritical and critical regime. In the supercritical range the random environment is less dominant, and both classes of processes share quite a few properties. Indeed, Tanny \cite{Tanny1988} derived the Kesten-Stigum theorem for GWPs in the random environment setup, see also \cite{Hambly1992}. This correspondence may be observed also for finer asymptotics of supercritical BPREs, as derived in \cite{Bansaye2014,Grama2017,Huang2014}. Thus, one may wonder whether Haldane's asymptotics for slightly supercritical GWPs, also transfers to BPREs. As we shall see, this is for iid environments largely, but not completely true.

Thus, let us consider a sequence $(\mathsf Z_{N},N\ge 1)$ of BPREs. All linked quantities are assigned to the index $N$, like the survival probabilities $\pi_N$, the random means $\mathsf M_N$, the generic offspring number $\mathsf U_N$ and so forth (we use sans-serif letters for random terms, which are designated to  carry the index $N$, like generic variables). In particular,  let the numbers $\eps_N$ and $\nu_N$ be given by the equations
\begin{align}
	\mathbb E[\mathsf M_N]=1+ \eps_N, \quad  \text{Var}(\mathsf M_N)= \nu_N.
\end{align}
Following Haldane, we are concerned with the situation that the  $\eps_N$ form a positive null sequence and that the variance of $\mathsf U_N$ stabilizes as $N \to \infty$, i.e.
\begin{align}
	\text{Var}(\mathsf U_N)= \sigma^2 +o(1) \quad \text{ with } \sigma^2>0. \label{assumption xi}
\end{align}
Using $\varepsilon_N =o(1)$ this assumption may be equally expressed as
\begin{align}
	\mathbb E[\mathsf M_N^{(2)}]=\sigma^2 + o(1). \label{assumption MN}
\end{align}
Without further mentioning we require $\mathsf M_N>0$   a.s. for all $N$, since otherwise $\pi_N=0$ trivially.

Within this scenario different behaviour may arise. There is a broad supercritical area, where Haldane's asymptotics proves true. Here the conditional mean $\mathsf M_N$ follows largely its expectation $1+ \varepsilon_N$, but deviations of higher magnitude than $\varepsilon_N$ may occur. This area will be left behind only, if $\nu_N$ attains the same order as $\varepsilon_N$, in other words, if the random fluctuations of $\mathsf M_N$ are of order $\sqrt{\varepsilon_N}$, thus  notably surpassing $\eps_N$, the expected excess of $\mathsf M_N$ above $1$.  In this area  of transition $\pi_N$ falls below  Haldane's estimate, the subcritical range is adjacent and the random environment makes itself felt. Our result requires stronger assumptions than the classical moment conditions on $\log \mathsf{M}_N$ for branching processes in iid random environments. This is no surprise, since we are addressing properties of $\mathsf{M}_N$ itself.

\begin{thm} \label{Theorem Surv prob BPRE} For branching processes $\mathsf Z_N$, $N \ge 1$, in iid random environments,
	let $\eps_N > 0$ for all $N$ and $\eps_N \to 0$ as $N \to \infty$. 	
	Additionally to \eqref{assumption xi} assume
	\begin{align}
		\mathbb E[\mathsf U_N^{4+\delta}]=O(1), \quad \mathbb E[\mathsf M_N^{-4-\delta}] =O(1), \quad \mathbb E[\big|\mathsf M_N-\mathbb E[\mathsf M_N]\big|^{4+\delta}]=O(\nu_N^{2+\frac{\delta}{2}}) \label{assumption F'}
	\end{align}
	with some $\delta>0$. Then we have:
	\begin{itemize}
		\item[i)] If $\frac {\nu_N}{\varepsilon_N} \to 0$, then the survival probability obeys Haldane's asymptotics 
		\begin{align}
			\pi_N \sim \frac{2 \eps_N }{\sigma^2} \quad \text{ as } N \to \infty.
		\end{align}
		\item[ii)] If  $ \frac {\nu_N}{\varepsilon_N} \to \rho$ with $0<\rho< 2$, then
		\begin{align}
			\pi_N \sim \frac{(2-\rho) \eps_N}{\sigma^2} \quad \text{ as } N \to \infty.
		\end{align}
		\item[iii)] If  $ \frac {\nu_N}{\varepsilon_N} \to 2$, then
		\begin{align}
			\pi_N=o(\varepsilon_N) \quad \text{ as } N \to \infty.
		\end{align}
		\item[iv)] If   $\liminf\limits_{N\to \infty}\frac {\nu_N}{\varepsilon_N}>2$, then for large $N$ the process $\mathsf{Z}_N$ is subcritical, implying
		\begin{align}
			\pi_N = 0.
		\end{align}
	\end{itemize}		
\end{thm}
Under items i) and ii) we are concerned with supercritical, and under iv) with subcritical processes (following from Lemma \ref{Lemma Moments of F} ii) below).  Only in the border case iii) this matter remains undecided, and we may have both $\pi_N>0$ and $\pi_N=0$. Here, a  more precise  handling of $\pi_N$ is a  challenge. In case ii) the random environment $\mathsf V_N$ comes to the fore, for instance at
\begin{align}
	\pi(\mathsf V_N):= \lim_{k\to \infty} \mathbb P(Z_k>0 \mid \mathsf V_N),
\end{align}
the conditional probabilities of ultimate survival, given $\mathsf V_N$. Other than under i), it stays asymptotically random, with a limiting $\Gamma$-distribution (see Proposition \ref{Lemma convergence in prob} below).

The theorem's proof does not rely on the familiar fixed point characterization of the survival probability, but follows a new strategy. We use an explicit representation for the survival probability in terms of the \emph{shape function} as well as the corresponding new techniques, which have been introduced in \cite{Kersting2020} for branching processes in varying environments. The derived expressions have a similar structure as perpetuities, known from a financial context. Thus, it is a major step of the proof to derive a limit theorem for perpetuities with asymptotically vanishing interest rates. To the best of our knowledge, this is the first time that a relation between branching processes in iid random environment and perpetuities occurs in the literature. Alsmeyer's publication \cite{Alsmeyer2021} uses this connection in the case of iid linear fractional environments, on arXiv it appeared almost simultaneously.

The paper is organized as follows. In Section \ref{Sec prelim} we recall the required notions on branching processes in a varying environment and establish the representation of the survival probability. Our result on perpetuities is given in Section \ref{Sec Perpetuities}, it might be of some interest on its own. There, we also provide a more detailed discussion of the relevant literature concerning perpetuities. The proof of the main theorem is given in Section \ref{Sec Proof}.

\section{An expression for the survival probability} \label{Sec prelim}
In this section we derive an expression for the survival probability $\pi$ of a BPRE, using ideas of \cite{Kersting2020}. As to the notation, we find it convenient to write the probability generating function of a probability  measure $p\in \mathcal P(\mathbb N_0)$ as
\begin{align}
	p(s)= \sum_{z=0}^{\infty} s^z p[z], \qquad 0\leq s \leq 1,
\end{align}
where $p[z]$ are the weights of $p$. 

Let us first look at the special case of a  deterministic environment $v=(p_1,p_2,...)$  of probability measures    $p_1,p_2, \ldots \in \mathcal P(\mathbb N_0)$ (which  is called a varying environment). For $k \le r$ the probability of survival in generation $r$, having one individual in generation $k$, is given by
\begin{align*}
	\pi_{k,r}(v) := \mathbb P(Z_r >0\mid Z_k=1).
\end{align*}
In order to express this probability more explicitly, define for $k \le r$ the probability generating functions
\begin{align*}
	p_{k,r}(s):= p_{k+1} \circ\cdots \circ p_{r}(s), \qquad 0 \leq s \leq 1,
\end{align*}
with the convention $p_{k,k}(s)=s$. Then, similar to classical GWPs we have for $k\le r$ the formula
\begin{align}
	\pi_{k,r}(v)= 1- p_{k,r}(0). \label{surv prob and gen function}
\end{align}
In order to expand the right hand term, we define as in \cite{Kersting2020} for any  $p\in \mathcal P(\mathbb N_0)$ with positive finite mean $m=p'(1)$  the shape function $\varphi_p :[0,1) \to \R$ via the equation
\begin{align}
	\frac{1}{1-p(s)}=\frac{1}{m(1-s)}+\varphi_p(s). \label{Def. Shape function}
\end{align}
Note that $\varphi_p$ is a non-negative function and can be extended continuously to $[0,1]$ by setting 
\begin{align}
	\varphi_p(1):= \frac{p''(1)}{2p'(1)^2}.
\end{align}
Letting $\varphi_k=\varphi_{p_k}$ be shape the function belonging to $p_k$, we obtain iteratively
\begin{align*}
	\frac{1}{1-p_{0,r}(s)} &=\frac{1}{p_1'(1)(1-p_{1,r} (s))} + \varphi_1(p_{1,r}(s))= \cdots = \\
	&=\frac{1}{p_1'(1)\cdots p_r'(1) (1-s)} +\sum_{k=1}^{r} \frac{\varphi_{k} ( p_{k,r}(s))}{p_1'(1)\cdots p_{k-1}'(1)},
\end{align*}
and for $s=0$ by  \eqref{surv prob and gen function}  
\begin{align}
	\frac{1}{\pi_{0,r}(v)}= \frac{1}{m_1\cdots m_r} 
	+\sum_{k=1}^{r} \frac{\varphi_{k} (1-\pi_{k,r}(v))}{m_1\cdots m_{k-1}} 
	\label{surv prob as sum}
\end{align}
with $m_k=p_k'(1)$. 

For an iid random environment $\mathsf V=(P_1,P_2, \ldots)$, we write $M_k, M^{(2)}_k$ and $\Phi_k(s)=\varphi_{P_k}(s)$ for the mean, second factorial moment and shape function of the random probability measure $P_k$. Note that $P_k(S)= \sum_{z=0}^\infty P_k[z]S^z$ and, consequently, $\Phi_k(S)$ are  well-defined random variables for any random variable $S$ taking values in $[0,1]$.  We set
\begin{align}
	F_k^-:=\Phi_k(0)=\frac 1{1-P_k[0]}-\frac 1{M_k} ,\qquad F_k^+:=\Phi_k(1)= \frac{M^{(2)}_k}{2 M_k^2}, \qquad k \in \N.
\end{align}
Again note that the $F_k^\pm$ are iid copies of some $\mathsf F^{\pm}$. 
Throughout this work we make use of the fact that it holds
\begin{align}
	\frac{1}{2} F_k^- \leq\Phi_k(S) \leq 2 F_k^+ \quad  \text{ a.s.} \label{eq:bound Phi}
\end{align}
for any random variable $0\le S \le 1$. This estimate follows via conditioning on $P_k$ from \cite[Lemma 1]{Kersting2020}, where it was derived for a varying environment.
We set
\begin{align}
	\mu_r = \prod_{k=1}^{r} M_k, \quad \mu_0 =1.
\end{align}
Now \eqref{surv prob as sum} reads
\begin{align}
	\frac{1}{\pi_{0,r}(\mathsf V)}= \frac{1}{\mu_r} 
	+\sum_{k=1}^{r} \frac{\Phi_{k} (1-\pi_{k,r}(\mathsf V))}{\mu_{k-1}} \quad \text{a.s.}
	\label{surv prob 2}
\end{align}
In this expression we are going to take the limit $r\to \infty$, in order to proceed to 
\begin{align*}
	\pi_{k,\infty}(\mathsf V) := \lim_{r\to \infty} \pi_{k,r}(\mathsf V) ,
\end{align*}
the probability of ultimate survival with $1$ individual at generation $k$, given the environment $\mathsf V$. We set
\begin{align}
	\pi(\mathsf V):= \pi_{0,\infty}(\mathsf V).
\end{align}

\begin{prop} \label{Prop expression survival prob}
	Let $\mathsf Z=(Z_k, k \geq 0)$ be a branching process in the random environment $\mathsf V=(P_1,P_2,..)$ consisting of independent copies of $\mathsf P$. Assume $0<\e{\log\mathsf  M}< \infty$ and $\e{\log^+ \mathsf M^{(2)}}< \infty$. Then the conditional survival probability can be expressed as
	\begin{align}
		\pi(\mathsf V)= \frac{1}{\mathsf X}, \label{prob and expec combined}
	\end{align}
	where a.s. 
	\begin{align}
		\mathsf X:= \sum_{k=1}^{\infty} \frac{\Phi_{k}(1-\pi_{k,\infty}(\mathsf V))}{\mu_{k-1}} < \infty.
	\end{align}
\end{prop}
\begin{rem} In particular we have $\pi= \mathbb E[1/X] >0$. 
	The assumptions in Proposition \ref{Prop expression survival prob} are slightly stronger than the classical requirements for a  positive probability of ultimate survival, see \cite{Smith1969, Tanny1977}. However, these publications contain no representation as \eqref{prob and expec combined}.	We note that Proposition \ref{Prop expression survival prob} holds equally for a stationary and ergodic random environment.
\end{rem}
\begin{proof}
	Observe that $\mu_r \to \infty$ a.s., since 
	\begin{align}
		\mu_r = \exp\big( \sum_{k=1}^{r} \log M_k \big) = \exp \left( r \e{ \log \mathsf M} + o(r)\right), \label{limit mu}
	\end{align}
	by the strong law of large numbers.	Furthermore, we have in the limit $r \to \infty$
	\begin{align}
		\Phi_{k} (1-\pi_{k,r}(\mathsf V)) \to \Phi_{k} (1-\pi_{k,\infty}(\mathsf V)) \quad \text{ a.s.} \label{limit of phi}
	\end{align}
	by continuity. Thus, \eqref{surv prob 2} yields
	\begin{align}
		\frac{1}{\pi(\mathsf V)} = \sum_{k=1}^{\infty} \frac{\Phi_{k} (1-\pi_{k,\infty}(\mathsf V))}{\mu_{k-1}}, \label{X geq 1}
	\end{align}
	provided that we can justify the interchange of limits. By \eqref{eq:bound Phi} we have 
	\begin{align}
		\Phi_{k} (1-\pi_{k,\infty}(\mathsf V)) \leq 2 F_k^+. 
	\end{align}
	Thus, letting
	\begin{align}
		\mathsf Y:=\sum_{k=1}^{\infty} \frac{F_{k}^+}{\mu_{k-1}},
	\end{align}
	it is sufficient in view of dominated convergence to prove $\mathsf Y<\infty$ almost surely. Note that $ F_k^+ \le \exp( \log^+ M^{(2)}_k -2 \log M_k)$, and by our assumptions and the strong law of large numbers we have a.s. $\log M_k = o(k)$ and 
	$\log^+M^{(2)}_k=o(k)$, therefore a.s. $F_k^+ = e^{o(k)}$. Together with \eqref{limit mu} it follows a.s. $\mathsf Y<\infty$.\end{proof}
\noindent Later, we approximate the random variable $\mathsf X$ by $\mathsf Y$. There, we rely on the following estimate.

\begin{lem} \label{complex approx}
	For an iid random environment $\mathsf V=( P_1,P_2,\ldots)$ we have for any $0<\eta \le 1/2$
	\begin{align}
		|\Phi_k(1)&- \Phi_k(1-\pi_{k,\infty}(\mathsf V))| \\
		&\le \eta (( M_k^{(2)})^2 +  M_k^{(2)} + \mathbb E[  U_{1,k}^4 \mid \mathsf V]) + 2 F_k^+ \1_{\{\pi_{k,\infty}(\mathsf V)> (\eta  M_k)^3 \text{ or } \eta  M_k >1\}}.
		\label{difference XY}
	\end{align}
\end{lem}

\begin{proof}  By conditioning on $\mathsf V$ we reduce our claim to the case of a varying environment. Hence, we may resort to \cite[Lemma 2]{Kersting2020} yielding
	\begin{align}
		|\Phi_k(1)- \Phi_k(S)| \le 2 \frac{( M_k^{(2)})^2}{ M_k^3}(1-S)+ 2 D \frac{M_k^{(2)}}{ M_k^2}(1-S)+ \frac{2}{ M_k^2}\mathbb E[ U_{1,k}^2 ;  U_{1,k}\ge D+1\mid \mathsf V],
	\end{align}
	where  $0\le S\le 1$ and $D\in \{1,2,\ldots\}$ are allowed to be any functions of $\mathsf V$. Choose $S= 1-\pi_{k,\infty}(\mathsf V)$ and $D=\lfloor \zeta^{-1} \rfloor$ with $\zeta = \eta M_k$ and some $0<\eta \le 1/2$. Then we obtain on the event $\{\pi_{k,\infty}(\mathsf V) \le \zeta^{3} \le 1\}=\{\pi_{k,\infty}(\mathsf V)\le \zeta^3,  D \ge 1\}$
	\begin{align}
		|\Phi_k(1)- \Phi_k(1-\pi_{k,\infty}(\mathsf V))| &\le 2 \frac{( M_k^{(2)})^2}{ M_k^3}\zeta^3+ 2  \frac{M_k^{(2)}}{ M_k^2}\zeta^2 + 2\frac{\zeta^2}{M_k^2}\mathbb E[U_{1,k}^4 \mid \mathsf V]\\
		&\le \eta(( M_k^{(2)})^2+M_k^{(2)}+ \mathbb E[U_{1,k}^4 \mid \mathsf V]).
	\end{align}
	On the complementary event $\{1-\pi_{k,\infty}(\mathsf V) \le \zeta^{3} \le 1\}^c$ we bound by means of \eqref{eq:bound Phi} leading to \eqref{difference XY}.
\end{proof}

\section{On perpetuities with small interest rates} \label{Sec Perpetuities}
Let $(A_k,B_k)$, $k \geq 1$, be independent copies of the random pair $(\mathsf A,\mathsf B)$,  where $\mathsf A$ and $\mathsf B$ are non-negative random variables  with finite means, and $\mathsf A$ is non-degenerate. Set $C_k:=A_1 A_2 \cdots A_k$ for $k \ge 1$, $C_0=1$, and consider the series
\begin{align}
	\mathsf Y:=\sum_{k=1}^{\infty } B_k C_{k-1},
	\label{series}
\end{align}
which in a financial context is called a \emph{perpetuity}, see \cite{Alsmeyer2009} and the literature cited therein. We allow $\mathsf Y$ to take the value $\infty$, see Remark \ref{rem:almost sure convergence Y}. The random variable $\mathsf Y$ fulfils a distributional recursion, the \emph{annuity equation} \cite{Buraczewski2016}
\begin{align}
	\mathsf Y \stackrel{d}{=} \mathsf A \mathsf Y+\mathsf B, \label{annuity equation}
\end{align}
where on the right-hand side $(\mathsf A,\mathsf B)$ and $\mathsf Y$ are assumed to be independent.

In what follows we study the limiting behaviour of a sequence of perpetuities 
\begin{align*}
	\mathsf Y_N\stackrel d= \mathsf A_N\mathsf Y_N+ \mathsf B_N, \quad N\ge 1,
\end{align*}
as the expectation of $\mathsf A_N$ tends to $1$ (which in a financial setting corresponds to asymptotically vanishing interest rates).  Let the  numbers $\alpha_N < 1$,  $\upsilon_N>0$ and  $\beta_N>0$, $ N\ge 1$, be given by
\begin{align}
	\mathbb E[\mathsf A_N]=1-\alpha_N, \quad \text{Var}(\mathsf A_N)=\upsilon_N,\quad \mathbb  E[\mathsf B_N] = \beta_N .
\end{align}
The asymptotic behaviour of $\mathsf Y_N$ is dictated by the expectation and variance of $\mathsf A_N$, whereas the expectation of $\mathsf B_N$ acts just as a scaling factor.
Depending on the asymptotic value of the ratio $\alpha_N/\upsilon_N$ we observe different limiting distributions for rescaled versions of $\mathsf Y_N$. 
\begin{thm} \label{Thm convergence perpetuity} Assume that as $N\to \infty $
	\begin{align}
		\alpha_N \to 0,  \quad \upsilon_N \to 0, \quad \beta_N\to \beta
	\end{align}
	with $0<\beta <\infty$, furthermore $\mathbb E[{|\mathsf A_N-1|^{2+\delta}}]= o(|\alpha_N|+\upsilon_N)$ and  $\mathbb E[{\mathsf B_N^{1+\delta}}]=O(1)$   for some \mbox{$\delta>0$}. Additionally, assume
	\begin{align}
		\frac{\alpha_N}{\upsilon_N} \to \gamma \quad \text{ with } -\frac 12\le \gamma \leq \infty.
	\end{align}
	Then we have: 
	\begin{itemize}
		\item [i)] If $\gamma=\infty$, then $\alpha_N \mathsf Y_N\to \beta$ in probability, as $N \to \infty$.
		\item [ii)] If $-1/2 <\gamma <\infty$, then $\upsilon_N \mathsf Y_N$ is asymptotically inverse $\Gamma$-distributed, with density\\ $cx^{-a -2} e^{-b/x} dx$, where $(a,b)=(2\gamma,2\beta)$ and $c= (2\beta)^{2\gamma+1}/\Gamma(2\gamma+1)$.
		\item [iii)] If $\gamma = -1/2$, then $\upsilon_N \mathsf Y_N \to \infty$ in probability, as $N \to \infty$.
	\end{itemize}
\end{thm}
Note that,  as long as $\gamma\neq0$,  also the $|\alpha_N|$ may serve as scaling factors. This might appear more natural, since for positive $\alpha_N$ we indeed have $\mathbb E[\mathsf Y_N]=\beta_N/\alpha_N$. However, for $\gamma=0$ they fail to be  usable for this purpose.

\begin{rem}\label{rem:almost sure convergence Y}{\em 
		The random variables $\mathsf Y_N$ are well defined, since we allow the value $\infty$. Either $\mathsf Y_N<\infty$ a.s. or $\mathsf Y_N=\infty$ a.s., by a $01$-law. The second case may occur in under iii) of Theorem \ref{Thm convergence perpetuity}. In both cases $\mathsf Y_N$ fulfils the annuity equation, since $\mathsf{Y}_N=\infty$ implies a.s. $\mathsf{A}_N>0$.
		
		Our theorem does not require any assumption on the convergence or divergence of the series in \eqref{series}, nevertheless there is a close connection. The classical criterion of Vervaat \cite{Vervaat1979} states that  in \eqref{series} a.s. convergence holds in case of $\mathbb E[\log \mathsf A] < 0$, and a.s. divergence in case of $\mathbb E[\log \mathsf A] > 0$ (for an ultimate criterion see \cite{Goldie2000}). In our case we have $\mathsf A_N\approx 1$, thus $\log \mathsf A_N \approx (\mathsf A_N-1) - (\mathsf A_N-1)^2/2 $ and typically $\mathbb E[\log \mathsf A_N] \approx - \alpha_N- \upsilon_N/2$. 
		
		Hence, for $\gamma >-1/2$ we expect a.s. convergence for large $N$, which will indeed result from Lemma \ref{Lemma tightness} below. Note that we may have a.s. convergence, even if $\alpha_N \le 0$ for all $N$, provided that the variance $\upsilon_N$ stays sufficiently large. Here $\E [\log \mathsf{A}_N] <0$, but $\E [\mathsf A_N]>0$ and consequently $\mathbb E[ \mathsf Y_N]=\infty$.
		
		On the other hand, for $\gamma < -1/2$ we typically have $\mathsf Y_N=\infty$ a.s. for large $N$. This range is of little interest in our context, thus we leave it aside in our theorem and focus on  the border  case $\gamma = -1/2$. Then a.s. convergence as well as a.s. divergence of \eqref{series} may occur.}
\end{rem}

\begin{rem}{\em  In the literature several papers address the limiting behaviour of perpetuities. A result matching to our Theorem~\ref{Thm convergence perpetuity}~ii) is contained in Dufresne's seminal paper   \cite{Dufresne1990}, see Proposition~4.4.4. therein. It requires stronger assumptions and aligns to ours only in case that $B$ is a.s. constant for all $N$. We note that Dufresne's proof rests on an invalid argument (namely, that distributional convergence in the Skorohod sense of some stochastic processes $(Z_n (t), t \geq 0)$  to a process $(Z(t), t \geq 0)$ implies convergence of $Z_n(\infty)=\lim_t Z_n(t)$ to $Z(\infty)=\lim_t Z(t)$ in distribution, provided the a.s. existence of these limits. In order to validate  Dufresne's proof it would be necessary to show that the random variable $T$, introduced on top of page 62, does not depend on $n$).
		
		Blanchet and Glynn \cite{Blanchet2022} as well as Iksanov et al. \cite{Iksanov2021} present results which belong to the range of part i) of our Theorem. Besides laws of large numbers they derive advanced approximations to the normal distribution. In another contribution Iksanov et al. \cite{Iksanov2023} provide refined asymptotics for the scaled logarithm of perpetuities in a range where $\alpha_N$ and $\upsilon_N$ would not converge to $0$. 
		
		Under the condition $\E [ \log \mathsf{A}]\geq 0$, Hitczenko and Weso\l owski \cite{Hitczenko2011} consider the convergence of rescaled partial sums of $\mathsf{Y}$.
		
		As a side remark, we note that equation \eqref{annuity equation} can be viewed as an equation for the stationary distribution of some real-valued Markov chain $(W_n, n \geq 0)$.  General results in this spirit have been obtained by Borovkov and Korshunov in \cite{Borovkov1992,Korshunov1994}, they show that  the limiting stationary distribution of sequences of Markov chains is $\Gamma$-distributed in a certain asymptotic regime. Thus, one might wonder, if our Theorem \ref{Thm convergence perpetuity} fits into this framework in the sense that the Markov chains $(1/W_n,  n \geq 0)$ can be integrated therein. However, it is readily checked that there is no match of the respective asymptotic regimes. Indeed, our result might give rise to another general result on asymptotic stationary distributions of Markov chains.
	}
\end{rem}

Let us turn to the proof of Theorem \ref{Thm convergence perpetuity}. We shall establish convergence of the corresponding Laplace transforms and characterize the limit by a  second order linear differential equation of singular type, related to the Bessel differential equation.  This approach necessitates that the terms $\mathsf A$ and $\mathsf B$ are non-negative. We note that a corresponding differential equation for characteristic functions is hardly available. This would require the existence of the second moment, which for the inverse $\Gamma$-distribution  is  in general not at disposal. Thus, if one would like to overcome the assumption of non-negativity, a different approach  seems to be needed.

We prepare the proof by  three lemmata. Let $\tau_N$ denote either $|\alpha_N|$ or $\upsilon_N$, and let 
\begin{align}
	\ell_N(\lambda):=\mathbb E[\exp(-\lambda \tau_N \mathsf Y_N)]  , \quad  \lambda \ge 0,
\end{align}
be the Laplace transform of $\tau_N \mathsf Y_N$. Recall, that $\ell_N$ is arbitrarily often differentiable at any $\lambda >0$. We set $e^{-\infty}=0$, thus $\ell_N(\lambda)=0$ for all $\lambda>0$ if a.s. $\mathsf Y_N=\infty$. This case causes no problems in the following proofs.

\begin{lem} \label{Lemma Laplace transform B and A}	
	Under the assumptions of Theorem \ref{Thm convergence perpetuity} we have for any $\lambda >0$ as $N\to \infty$
	\begin{align}
		\frac 12 \upsilon_N\lambda \ell_N''(\lambda)-\alpha_N  \ell_N'(\lambda) - \beta \tau_N  \ell_N(\lambda) = o(|\alpha_N|+\upsilon_N).
		\label{DGLN}
	\end{align}
\end{lem}
\begin{proof} Let $\ell(\lambda)=\mathbb E[e^{-\lambda \tau \mathsf Y}]$ be the Laplace transform of $\tau \mathsf Y$ with $\tau >0$ and with a single perpetuity $\mathsf Y$.   By means of equation \eqref{annuity equation} and independence of $(\mathsf A,\mathsf B)$ and $\mathsf Y$ we have for any $\lambda \ge 0$
	\begin{align}
		\ell(\lambda)=  \mathbb E[\mathbb E[e^{-\lambda\tau(\mathsf A\mathsf Y+\mathsf B)} \mid (\mathsf A,\mathsf B)]] = \mathbb E[e^{-\lambda\tau \mathsf B}\ell(\lambda \mathsf A)].
		\label{annuity equation2}
	\end{align}	
	We approximate the right-hand expectation by means of Taylor expansions of its integrand, which is done in three steps. Let $\eta>0$.
	
	i)  First, by restricting expectations onto the event $\{\tau \mathsf B\le \eta\}$ and its complement, we have
	\begin{align}
		\big| \mathbb E[e^{-\lambda \tau \mathsf B} &\ell(\lambda \mathsf A)]- \mathbb E[(1-\lambda \tau \mathsf B)\ell(\lambda \mathsf A)] \big|
		\\& \le \mathbb E [|e^{-\lambda \tau \mathsf B}-(1-\lambda \tau \mathsf B)|\ell(\lambda \mathsf A);\tau \mathsf B\le \eta]
		+\mathbb E [|e^{-\lambda \tau \mathsf B}-(1-\lambda \tau \mathsf B)|\ell(\lambda \mathsf A);\tau \mathsf B> \eta]
		\\& \le \mathbb E [\lambda^2 \tau^2 \mathsf B^2 ; \tau \mathsf B\le \eta]+ \E [2+\lambda \tau \mathsf B;  \tau \mathsf B>\eta] 
		\\& \le \lambda^2 \eta \tau \mathbb E[\mathsf B] +  \big( \frac 2{\eta^{1+\delta}} + \frac \lambda{\eta^\delta}\big) \tau^{1+\delta}\mathbb E [\mathsf B^{1+\delta}]. \label{eq:two terms}
	\end{align}
	Now we take the dependence on $N$ into account. Let $\kappa >0$. By the assumptions of Theorem \ref{Thm convergence perpetuity} and by choosing $\eta$ sufficiently small, the first term in \eqref{eq:two terms} becomes smaller than $\kappa \tau_N$. Since $\tau_N^\delta \to 0$, also the second term in \eqref{eq:two terms} becomes smaller than $\kappa \tau_N$, if only $N$ is large enough. This entails
	\begin{align}
		\mathbb E[e^{-\lambda \tau_N \mathsf B_N} &\ell_N(\lambda \mathsf A_N)]=\mathbb E[ \ell_N(\lambda \mathsf A_N)]- \lambda \tau_N\mathbb E[  \mathsf B_N\ell_N(\lambda \mathsf A_N)] + o(\tau_N). \label{approx1}
	\end{align}
	
	ii) The  right-hand expectations in \eqref{approx1} are handled similarly, now by restricting them to the event $\{ |\mathsf A-1| \le \eta\}$ and its complement. For $0<\eta \le 1/2$, $\lambda >0$  we have 
	\begin{align*}
		\lambda \sup_{|a-1|\le \eta} |\ell'(a\lambda)| \le \mathbb E[\lambda \tau \mathsf Y e^{-(1-\eta)\lambda \tau \mathsf Y} ] \le \sup_{z\ge 0} z e^{- z/2 }.
	\end{align*}
	Hence, with some random $A'$ between $\mathsf A$ and 1,
	\begin{align*}
		\big|\mathbb E[\mathsf B\ell(\lambda \mathsf A)] -  \mathbb E[\mathsf B \ell(\lambda)]\big| 
		&\le  \mathbb E [\mathsf B |\ell'(\lambda A')|\lambda |\mathsf A-1|; |\mathsf A-1| \le \eta] +  \mathbb E[ 2\mathsf B; |\mathsf A-1| > \eta]\\
		&\le \lambda\eta \mathbb E[\mathsf B]\sup_{|a-1|\le \eta} |\ell'(a\lambda)| + 2 \mathbb E[\mathsf B^{1+\delta}]^{\frac1 {1+\delta}} \mathbb P (|\mathsf A-1| >\eta )^{\frac \delta{1+\delta}} \\
		&\le \eta \mathbb E[\mathsf B] \sup_{z\ge 0} z e^{- z/2 } + 2 \mathbb E[\mathsf B^{1+\delta}]^{\frac1 {1+\delta}} \Big(\frac {\mathbb E[(\mathsf A-1)^{2} ] }{\eta^{2}}\Big)^{\frac \delta {1+\delta}}.
	\end{align*}
	
	Regarding the dependence on $N$, by the theorem's assumptions, by suitably adapting $\eta$ and by noting $\EE{(\mathsf{A}_N-1)^2}=\upsilon_N + \alpha_N^2 \to 0$, this expression can be, with increasing $N$, made smaller than any $\kappa >0$. In other terms:
	\begin{align}
		\mathbb E[\mathsf B\ell_N(\lambda \mathsf A_N)] =  \beta \ell_N(\lambda) + o(1) . \label{approx2}
	\end{align}
	
	iii) Finally, we have for $\lambda>0$ and $0<\eta \le 1/2$ with some $A'$ between $\mathsf A$ and 1
	\begin{align*}
		\Big|\mathbb E [\ell(\lambda \mathsf A)]&- \mathbb E[ \ell(\lambda)+ \ell'(\lambda)\lambda(\mathsf A-1) + \frac 12 \ell''(\lambda)\lambda^2(\mathsf A-1)^2]\Big| \\
		&\le \mathbb E[ |\ell'''(\lambda \mathsf A')|\lambda^3 |\mathsf A-1|^3; |\mathsf A-1| \le \eta)]\\
		&\qquad \mbox{}+ \mathbb E[2+ |\ell'(\lambda)|\lambda |\mathsf A-1|+ \ell''(\lambda)\lambda^2|\mathsf A-1|^2 ; |\mathsf A-1|>\eta] \\
		&\le \eta \mathbb E[(\mathsf A-1)^2] \sup_{|a-1|\le 1/2} |\ell'''(a\lambda)|+ \Big(\frac 2{\eta^{2+\delta}}+ \frac{|\ell'(\lambda)|}{\eta^{1+\delta}}+ \frac{\ell''(\lambda)}{\eta^{\delta} }\Big) \mathbb E[|\mathsf A-1|^{2+\delta} ].
	\end{align*}
	Regarding the dependence on $N$, the derivatives of $\ell$, including the preceding supremum, stay again bounded with $N$. Further,  we have $\EE{(\mathsf{A}_N-1)^2}=O( \upsilon_N + |\alpha_N|)$ and $\mathbb E[|\mathsf A_N-1|^{2+\delta}]= o(\upsilon_N+|\alpha_N|)$ by assumption.  Thus, letting $\kappa >0$ and adapting $\eta$ once more, the above right-hand term becomes smaller than $\kappa(|\alpha_N|+\upsilon_N)$ for large $N$. It follows
	\begin{align}
		\mathbb E [\ell_N(\lambda \mathsf A_N)]=\ell_N(\lambda)+  \mathbb E[ \ell_N'(\lambda)\lambda(\mathsf A_N-1) + \frac 12 \ell_N''(\lambda)\lambda^2(\mathsf A_N-1)^2]+ o(|\alpha_N|+\upsilon_N).
		\label{approx3}
	\end{align}
	Combining equations \eqref{approx1}, \eqref{approx2} and \eqref{approx3} we arrive at
	\begin{align*}
		\mathbb E[e^{-\lambda \tau_N \mathsf B_N} \ell_N(\lambda \mathsf A_N)]= \ell_N(\lambda)- \alpha_N \lambda \ell_N'(\lambda)+ \frac 12 \upsilon_N\lambda^2 \ell_N''(\lambda) - \beta \tau_N \lambda \ell_N(\lambda) + o(|\alpha_N|+\upsilon_N). 
	\end{align*}
	Inserting this approximation into \eqref{annuity equation2} yields our claim. \end{proof}

The next lemma applies under weaker assumptions on the sequence $(A_k,B_k)$, which will be of advantage later.
\begin{lem}
	\label{Lemma tightness}
	Assume that the sequence $(A_k)$ consists of independent copies of the non-negative random variable $\mathsf A$, and that $(B_k)$ contains non-negative random variables with identical finite mean $\mathbb E[\mathsf B]>0$. Then, if $\mathbb E[\mathsf A^u]<1$ for some $u>0$, we have for any $c>0$
	\begin{align}
		\mathbb P\big( (1- \mathbb E[\mathsf A^u]^{\frac 1u}) \mathsf Y > c \mathbb E[\mathsf B]\big) \le 2 c^{-\frac u{1+u}}. 
	\end{align}
\end{lem}

\begin{proof} The cases $\EE{\mathsf{A}}=0$ and $\EE{\mathsf B}=0$ cause no problems, thus we assume $\EE{\mathsf{A}}>0$ and $\EE{\mathsf{B}}>0$. Consider  the process $M_k=C_k^u/\mathbb E[\mathsf A^u]^{k}$, $k\ge 0$, which is a non-negative martingale. For the stopping time $T=\min\{k \ge 0: M_k > a\}$  we obtain by means of the optional stopping theorem  for any $t \in \mathbb N$
	\begin{align}
		1=M_0=\mathbb E[M_{T\wedge t}] \geq a \mathbb P(T \leq t),
	\end{align} 
	thus
	\begin{align}
		\mathbb P(T < \infty)\leq 1/a. \label{Stopping time T}
	\end{align}
	On the event $\{T=\infty\}$ we have 
	$C_k  \leq a^{\frac 1u} \mathbb E[\mathsf A^{u} ]^{\frac{k}{u}}$ for all $k \ge 1$.
	This yields for any $\tau >0$  
	\begin{align}
		\mathbb P(\tau \mathsf Y> c \mathbb E[\mathsf B]) &\leq  \frac 1a+ \frac{\tau}{c \mathbb E[\mathsf B]}\mathbb E[\mathsf Y;T=\infty]  \\&\leq \frac{1}{a} +\frac{\tau}{c\mathbb E[\mathsf B]} \sum_{k=1}^{\infty} a^{\frac{1}{u } }\mathbb E\big[B_k\mathbb E[\mathsf A^{u} ]^{\frac{k-1}{u}}\big] \\ &= \frac{1}{a} +\frac{\tau a^{\frac{1}{u}}}  {c(1-\mathbb E[\mathsf A^{u} ]^{\frac 1u})}.
		\label{abschaetzung}
	\end{align}
	Letting $a=c^{u/(1+u)}$ and $\tau= 1-\mathbb E[\mathsf A^{u} ]^{\frac 1u}$ yields the claim.\end{proof}

\begin{rem} {\em  The previous inequality shows that the tail probabilities of the random variable $\mathsf Y$ decrease at a polynomial rate. This issue attracted some attention in the literature. Sophisticated treatments by Goldie \cite{Goldie1991} and others \cite{Collamore2013,Buraczewski2016} show that under additional assumptions we indeed have \mbox{$\mathbb P(\mathsf Y>c) \sim dc^{-\xi}$} with  constants $d, \xi>0$, where $\xi$ fulfils $\mathbb E[\mathsf A^\xi]=1$. Moreover, Collamore and Vidyashankar  presented an upper bound  of the form $\mathbb P(\mathsf Y>c) \le d'c^{-\xi}$ with some $d'>d$ \cite[Proposition 2.1]{Collamore2013}, or $\mathbb P((d')^{-1/\xi}\mathsf Y> c ) \le c^{-\xi}$. With our approach we underestimate the correct exponent $\xi$, however this is not our objective. We aim at suitable scaling factors for $\mathsf Y$. Using the formula 
		$\mathbb P((d')^{-1/\xi}\mathsf Y> c ) \le c^{-\xi}$ would require to sufficiently decrease $d'$, and at the same time get a handle on $\xi$, a formidable task. In contrast, our Lemma \ref{Lemma tightness} provides the clear-cut, explicit scaling factor $1- \EE{\mathsf A^u}^{1/u}$. As we shall see, it has the right magnitude. (As to exponential tail bounds, compare \cite{Hitczenko2009}).}
\end{rem}

\begin{lem} \label{lemma DGL solution}
	Let $a\in \mathbb R$, $b>0$. Then the linear differential equation
	\begin{align}
		\lambda \ell''(\lambda)= a\ell'(\lambda)+ b\ell(\lambda) , \quad \lambda >0, 
		\label{DGL}
	\end{align}
	has the  solution
	\begin{align}
		l (\lambda)= c\int_0^\infty e^{-\lambda x}x^{-a-2} e^{-b/x} \, dx,
		\label{LapTrans}
	\end{align}
	which for $c>0$ is the Laplace transform of the measure on $\mathbb R^+$ with the density $cx^{-a-2} e^{-b/x}$ with respect to the Lebesgue measure. The Laplace transform of any other non-vanishing measure on $\mathbb R^+$ fails to solve equation \eqref{DGL}.
\end{lem}

\begin{proof} The differential equation can be solved by means of modified Bessel functions. Our approach is  more direct and elementary. By partial integration we have
	\begin{align*} \lambda \int_0^\infty e^{-\lambda x}x^{-a}e^{-b/x} \, dx = \int_0^\infty e^{-\lambda x} \big( -a x^{-a-1}+bx^{-a-2}\big)e^{-b/x}\, dx ,
	\end{align*}
	which  transforms directly into  equation \eqref{DGL} for the function $l$ from \eqref{LapTrans}.
	
	For the second claim note that $l(\lambda)\to 0$ as $\lambda \to \infty$. Consider some non-vanishing measure $\mu$ on $\mathbb R^+$ with finite $h(\lambda)=\int_0^\infty e^{-\lambda x}\, \mu(dx) $ for all $\lambda >0$. Suppose that $h$ solves \eqref{DGL}. Since the $n$-th derivative of $h$ is equal to $\pm \int_0^\infty x^n e^{-\lambda x}\, \mu(dx)$, we have $h'(\lambda), h''(\lambda) \to 0$ as $\lambda \to \infty$. From \eqref{DGL} it follows that as well $h(\lambda)\to 0$, as $\lambda \to \infty$ (meaning that $\mu$ has no atom at zero).
	
	Let $\lambda_1>0$. Since $h(\lambda_1)$ and $l(\lambda_1)$ are strictly positive, there is a $c>0$ in \eqref{LapTrans} such that $h(\lambda_1)=l(\lambda_1)$.  Thus the difference $d=h-l$  satisfies $d(\lambda_1)=0$ and $d(\lambda)\to 0$ as $\lambda \to \infty$. Moreover $d$ fulfils the equation \eqref{DGL}. Suppose that $d$ does not vanish on the interval $(\lambda_1,\infty)$. Then $d$ has  at some point $\lambda_2>\lambda_1$ a global maximum or minimum, in particular $d'(\lambda_2)=0$. In case of a maximum we have $d(\lambda_2)>0$ and $d''(\lambda_2)\le 0$, which  contradicts \eqref{DGL}. The case of a minimum is analogue, thus $d$ has to vanish on the whole interval $(\lambda_1,\infty)$. Since $\lambda_1>0$ is arbitrary, we conclude that $h=l$, and our claim follows by uniqueness of Laplace transforms.\end{proof}
\begin{proof}[Proof of Theorem \ref{Thm convergence perpetuity}]
	We prove convergence of the Laplace transform of $\alpha_N \mathsf Y_N$ or $\upsilon_N \mathsf Y_N$ towards the Laplace transform of the limiting distributions. Note that convergence in distribution implies not only convergence of the Laplace transforms, but also  convergence of their derivatives at points $\lambda >0$, being of the form $\pm \mathbb E[(\tau_N\mathsf Y_N)^n e^{-\lambda \tau_N\mathsf Y_N}]$.
	
	In order to obtain tightness we shall apply Lemma \ref{Lemma tightness}. For $0<u<1$ and $0<\eta <1$ we have
	\begin{align*}
		\Big| \mathbb{E} \big[\mathsf A^u&-\big(1+u(\mathsf A-1)+ \frac{u(u-1)}2 (\mathsf A-1)^2 \big) \big] \Big|\\
		&\le\mathbb E \Big[(\mathsf A')^{u-3}|\mathsf A-1|^3;|\mathsf A-1|\le \eta\Big] + \mathbb E\Big[|\mathsf A-1|^u + |\mathsf A-1|+ |\mathsf A-1|^2 ; |\mathsf A-1|>\eta \Big] \\
		&\le \eta (1-\eta)^{u-3} \mathbb E\big[(\mathsf A-1)^2\big]+ \Big(\frac{1}{\eta^{2+\delta-u}}+ \frac 1{\eta^{1+\delta}} + \frac 1{\eta^\delta} 
		\Big)\mathbb E\big[|\mathsf A-1|^{2+\delta}\big].
	\end{align*}
	Turning to the sequence $(\mathsf Y_N)$, let $\kappa >0$. As in the proof of Lemma \ref{Lemma Laplace transform B and A} iii),  by a suitable choice of $\eta$, the right-hand expression falls below $\kappa(|\alpha_N|+\upsilon_N)$ for large $N$. It follows
	\begin{align*}
		\mathbb E[\mathsf A_N^u] = 1 -u\Big(\alpha_N+\frac {1-u}2 \upsilon_N\Big) + o(|\alpha_N|+\upsilon_N),
	\end{align*}
	consequently
	\begin{align}
		1- \mathbb E[\mathsf A_N^u]^{\frac 1u} = \alpha_N+ \frac {1-u}2 \upsilon_N + o(|\alpha_N|+\upsilon_N) .
		\label{tightness approx}
	\end{align}
	i)  If $\gamma=\infty$, then $\upsilon_N=o(\alpha_N)$.  Then \eqref{tightness approx} turns into $1- \mathbb E[\mathsf A_N^u]^{1/u} \sim \alpha_N$, and  the sequence $(\alpha_{N} \mathsf Y_N)$ is tight in view of Lemma \ref{Lemma tightness}. Let $\alpha_{N'} \mathsf Y_{N'}$ be a subsequence converging in distribution and with limiting distribution $\ell(\lambda)$.  Dividing both sides by $ \alpha_{N'}$  in \eqref{DGLN} and taking $\upsilon_N=o(\alpha_N)$ into account we obtain the limiting differential equation
	\begin{align*}
		\ell'(\lambda)+\beta \ell(\lambda) =0, \qquad \ell(0)=1,
	\end{align*}
	which has the unique solution $\ell(\lambda)= e^{-\beta \lambda}$, the Laplace transform of the Dirac measure at point $\beta$. As is well-known, this implies part i) of Theorem \ref{Thm convergence perpetuity}.
	
	ii) Next assume $-1/2 < \gamma <  \infty$. Then, because of \eqref{tightness approx} $1- \mathbb E[\mathsf A_N^u]^{1/u} \sim (\gamma + \frac{1-u}2)\upsilon_N$, which becomes positive for large $N$, if $u$ is chosen sufficiently small. Here the sequence $(\upsilon_N \mathsf Y_N)$ is tight in view of Lemma \ref{Lemma tightness}. Consider a convergent subsequence $\upsilon_{N'} \mathsf Y_{N'}$. Choosing $\tau_N = \upsilon_N $ in \eqref{DGLN}  we obtain the limiting differential equation
	\begin{align}
		\frac{1}{2} \lambda  \ell''(\lambda)-\gamma  \ell'(\lambda) -\beta \ell(\lambda)  =0, \qquad \ell(0)=1.  
		\label{bessel dgl}
	\end{align}
	Lemma \ref{lemma DGL solution} yields that the limiting distribution is the inverse $\Gamma$-distribution, as stated.
	
	iii) In case of $\gamma = -1/2$  we  consider a convergent subsequence $\upsilon_{N'} \mathsf Y_{N'}$ within the extended compactified range  $[0, \infty]$. Then the limiting Laplace transform $\ell(\lambda)$ again solves \eqref{bessel dgl}, but the corresponding distribution  may now result in  a defective probability measure on $\mathbb R^+$. The measure given by \eqref{LapTrans} with $a=2\gamma=-1$ is no longer finite and thus impractical. From Lemma~\ref{lemma DGL solution} we see that there is no other non-vanishing measure on $\R^+$ at disposal. Therefore, the limiting distribution of $\upsilon_{N'} \mathsf Y_{N'}$ has to be concentrated at the point $\infty$, which implies our claim. \end{proof}

\section{Proof of Theorem \ref{Theorem Surv prob BPRE}} \label{Sec Proof}
In view of Proposition \ref{Prop expression survival prob} we determine the limit of $\mathbb E[1/\mathsf X_N]$ as $N\to \infty$, recall
\begin{align}\label{formula X}
	\mathsf X:= \sum_{k=1}^{\infty} \frac{\Phi_{k}(1-\pi_{k}(V))}{\mu_{k-1}} < \infty \quad  \text{ a.s.}
\end{align} 
To this end we shall prove that this expectation may  be replaced by $\mathbb E[1/\mathsf Y_N]$, where
\begin{align} \label{formula Y}
	\mathsf Y:= \sum_{k=1}^{\infty} \frac{F_k^+}{\mu_{k-1}}.
\end{align}
This random random variable has the form \eqref{series}, thus we may apply Theorem \ref{Thm convergence perpetuity} to obtain the limiting distribution of the scaled $\mathsf Y_N$, respectively $\mathsf X_N$. In order to switch to expectations, we shall show uniform integrability of the scaled $1/\mathsf X_N$. 
Then it remains to determine the corresponding expectation of the limiting distribution. We prepare the proof  by several lemmata. 
\begin{lem}\label{Lemma Moments of F}
	Under the assumptions of Theorem \ref{Theorem Surv prob BPRE}, if $\nu_N=o(1)$ then we have
	\begin{itemize}
		\item[i)] $\EE{\mathsf{M}_N^{-u}}= 1-u\eps_N + \frac{u(u+1)}{2}\nu_N + o(\eps_N+\nu_N)$,  for $0\le u \le 2$,
		\item[ii)] $\EE{\log \mathsf M_N}= \eps_N- \frac 12 \nu_N +o(\eps_N+\nu_N)$,
		\item[iii)] $\EE{\frac{\mathsf M_N^{(2)}}{\mathsf{M}_N^2}}= \sigma^2+o(1)$.
	\end{itemize}  
\end{lem}

\begin{proof} i) Let $0< \eta < 1$. Similar as in the proof of Lemma \ref{Lemma Laplace transform B and A} we split the expectation $\EE{\mathsf{M}^{-u}}$ into its parts on the event $\{|\mathsf{M}-1| \leq \eta \}$ and the complement. Thus
	\begin{align}
		\big| \mathbb E[\mathsf M^{-u} ] - \mathbb E [1-u(\mathsf M-1)&+ \frac {u(u+1)}2 (\mathsf M-1)^2 ] \big| \\ 
		&\leq \mathbb E\big[ \frac{u(u+1)(u+2)}6(1-\eta)^{-u-3} |\mathsf M-1|^3; |\mathsf M-1|\le \eta \big] \\ 
		&\qquad \mbox{}+ \mathbb E\big[ \mathsf M^{-u}+1+u|\mathsf M-1| + \frac{u(u+1)}2 (\mathsf M-1)^2 ; |\mathsf M-1|>\eta \big]\\
		&\le 4\eta (1-\eta)^{-u-3} \mathbb E\big[(\mathsf M-1)^2] + \mathbb E[ \mathsf M^{-2u}]^{\frac 12} \mathbb P(|\mathsf M-1| >\eta)^{\frac 12} \\
		&\qquad \mbox{}+ \Big( \frac 1{\eta^{2+\delta}}+ \frac 2{\eta^{1+\delta}}+ \frac 3{\eta^\delta} \Big) \mathbb E\big[ |\mathsf M-1|^{2+\delta} \big].
	\end{align}
	Turning to the sequence $(\mathsf M_N)$ we obtain with $\delta>0$ similar as in   \eqref{assumption F'}
	\begin{align}
		\big| \mathbb E[\mathsf M^{-u} ] - (1- u\varepsilon_N+ \frac{u(u+1)}2(\nu_N+ \varepsilon_N^2)) \big| 
		\leq 4\eta (1-\eta)^{-u-3} (\nu_N +\varepsilon_N^2) + O(\nu_N^{1+\delta/2}+\varepsilon_N^{2+\delta}).
	\end{align}
	Let $\kappa>0$. Then,  because of $\varepsilon_N,\nu_N=o(1)$, there is an $\eta >0$  such that the right-hand expression is smaller than $\kappa (\varepsilon_N+\nu_N)$ for large $N$. In other terms: The right-hand expression is of order $o(\varepsilon_N+\nu_N)$, and our claim follows.
	
	ii) By the same line of argument, using $|\log x| \le x+x^{-1}$ for all $x>0$, we have for $0<\eta < 1$
	\begin{align}
		\big| \mathbb E[\log \mathsf M_N] - \mathbb E[ (\mathsf M_N-1)&- \frac 12(\mathsf M_N-1)^2 ]\big|\\
		& \le \mathbb E[ \frac 13 (1-\eta)^{-3} |\mathsf M_N -1|^3 ; |\mathsf M_N-1| \le \eta]\\ 
		&\qquad \mbox{}+ \mathbb E\big[ (\mathsf M_N+ \mathsf M_N^{-1}) + |\mathsf M_N-1| + \frac 12 (\mathsf M_N-1)^2 ; |\mathsf M_N-1|> \eta],
	\end{align}
	and our claim is confirmed in much the same vein as under i).	
	
	iii) 
	We have
	\begin{align}
		\Bigg|\EE{\frac{\mathsf M_N^{(2)}}{\mathsf M_N^2}} - \EE{\mathsf M_N^{(2)}} \Bigg| \le \EE{\Bigg| \mathsf M_N^{(2)} \left(\frac{1}{\mathsf M_N^2} -1 \right) \Bigg|}\leq \EE{\mathsf (M_N^{(2)})^2}^{\frac{1}{2}} \EE{\left(\frac{1}{\mathsf M_N^2} -1 \right)^2}^{\frac{1}{2}}.
	\end{align}
	The first term on the right-hand side is uniformly bounded by \eqref{assumption F'}. For the second term note that the integrand converges to $0$ in probability and is uniformly integrable due to \eqref{assumption F'}, which implies that the second term converges to $0$. This finishes the proof by noting $\E [\mathsf M^{(2)}_N ]= \sigma^2 + o(1)$. \end{proof}

\begin{lem}\label{Lemma bound phi(0)}
	Under the assumptions of Theorem \ref{Theorem Surv prob BPRE}, if $\nu_N=o(1)$, then there exists $\theta >0$ and $p_0>0$ such that for large $N$
	\begin{align}
		\mathbb P(\mathsf F_N^-> \theta) \geq p_0.
	\end{align}
\end{lem}

\begin{proof}
	From $\mathsf F^-=(1-\mathsf P[0])^{-1}-\mathsf M^{-1} \ge 1+\mathsf P[0]-\mathsf M^{-1}$ we obtain
	\begin{align}
		\mathbb P(\mathsf F^- > \theta) &\geq \mathbb P(\mathsf P[0] >  \theta +\mathsf M^{-1}-1)\\&\geq \mathbb P(\mathsf P[0] >  2\theta) -\mathbb P(1-\mathsf M^{-1}> \theta)\\
		&\geq \mathbb E[\mathsf P(0)] - 2\theta -\mathbb P(1-\mathsf M^{-1}> \theta), \label{geq p_0}
	\end{align}
	where in the last inequality we used the fact that $0 \le \mathsf P[0] \le 1$. On the other hand, for any $k>0$
	\begin{align}
		\mathbb E[(\mathsf U-1)^2;\mathsf U \le k] \le \mathbb P(\mathsf U=0) + (k-1)\mathbb E[\mathsf U-1; \mathsf U\ge 2] =  k\mathbb P(\mathsf U=0) + (k-1)\mathbb E[\mathsf M-1].
	\end{align}
	Using  $\mathbb E[\mathsf P(0)]=\mathbb P(\mathsf U=0)$ we arrive at the estimate
	\begin{align}
		\mathbb P(\mathsf F^- > \theta) \ge \frac 1k \mathbb E[(\mathsf U-1)^2;\mathsf U \le k] - \frac{k-1}k \mathbb E[\mathsf M-1]-2\theta -\mathbb P(1-\mathsf M^{-1}> \theta).
	\end{align}
	
	Turning to the sequence $(\mathsf Z_N)$ of branching processes we note that because of \eqref{assumption F'} the sequence $(\mathsf U_N^2)$ is uniformly integrable. Thus for $k$ sufficiently large we have for all $N$
	\begin{align}
		\mathbb E[(\mathsf U_N-1)^2;\mathsf U_N \le k] \ge \frac 12 \mathbb E[(\mathsf U_N-1)^2] \ge \frac 12 \text{Var}(\mathsf U_N) = \frac{\sigma^2} 2 +o(1).
	\end{align}
	Moreover, $\mathbb E[\mathsf M_N-1] =o(1)$ and $\mathbb P(1-\mathsf M_N^{-1}> \theta)=o(1)$ because of $\varepsilon_N=o(1)$ and $\nu_N=o(1)$. Thus we end up with
	\begin{align}
		\mathbb P(\mathsf F_N^- > \theta) \ge \frac{\sigma^2}{2k} -2\theta +o(1),
	\end{align}
	which implies our claim.\end{proof}
\begin{lem} \label{Lemma convergence Y}
	Under the assumptions of Theorem \ref{Theorem Surv prob BPRE},
	\begin{itemize}
		\item [i)] if $\frac{\nu_N}{\eps_N}\to 0$ as $N \to \infty$, then $\eps_N \mathsf{Y}_N$ converges in probability to $\frac{\sigma^2}{2}$,
		\item[ii)] if $\frac{\nu_N}{\eps_N}\to \rho $ as $N \to \infty$ with $0 < \rho <2$, then $\nu_N \mathsf{Y}_N$ is asymptotically inverse $\Gamma$-distributed with density $c x^{-a-2}e^{-b/x}$, where $(a,b)=\big( \frac{2(1-\rho)}{\rho}, \sigma^2  \big)$.
		\item[iii)] if $\frac{\nu_N}{\eps_N}\to 2$ as $N \to \infty$, then $\nu_N \mathsf{Y}^-_N \to \infty$ in probability, where $\mathsf{Y}^-:= \sum_{k=1}^{\infty} \frac{F_k^-}{\mu_{k-1}}$.
	\end{itemize}
\end{lem}
\begin{proof}
	We are going to apply Theorem \ref{Thm convergence perpetuity}. Note that $\nu_N=o(1)$ because of $\varepsilon_N=o(1)$. The random variables $\mathsf{Y}_N$ fulfil the distributional equation
	\begin{align}
		\mathsf{Y}_N \stackrel{d}{=} \frac{1}{\mathsf{M}_N} \mathsf{Y}_N + \mathsf{F}^+_N,
	\end{align}
	which is the annuity equation \eqref{annuity equation}
	with $\mathsf{A}_N=1/\mathsf{M}_N$ and $\mathsf{B}_N=\mathsf{F}^+_N$. We have to verify the assumptions of Theorem \ref{Thm convergence perpetuity}. Regarding the random variables $\mathsf B_N$,   Lemma~\ref{Lemma Moments of F}~iii) yields
	\begin{align}
		\EE{\mathsf{B}_N} = \frac{\sigma^2}{2}+o(1),
	\end{align}
	thus $\beta = \frac{\sigma^2}{2}$. Additionally, we have  for some $\delta>0$ 
	\begin{align}
		\EE{\mathsf B_N^{1+\delta}} \leq \EE{\left(\mathsf{M}_N^{(2)} \right)^{2+2\delta} }^{\frac{1}{2}} \EE{\mathsf{M}_N^{-4-4\delta}}^\frac 12 =O(1) ,\label{uniform integrability varphi}
	\end{align}
	by assumption    \eqref{assumption F'}. For $\alpha_N=1- \EE{ \mathsf{A}_N}$, Lemma \ref{Lemma Moments of F} i) yields
	\begin{align}
		\alpha_N = \eps_N - \nu_N + o(\eps_N),
	\end{align}
	and for $\upsilon_N = \VV{ \mathsf{A}_N}$
	\begin{align}
		\upsilon_N&= \EE{ \mathsf{M}_N^{-2}} - \EE{ \mathsf{M}_N^{-1}}^{2} = \nu_N + o(\eps_N) .
	\end{align}
	Finally, we have to confirm $\EE{|\mathsf{A}_N -1|^{2+\delta}}=o(|\alpha_N|+\upsilon_N)$. It holds for $\delta>0$ sufficiently small by  assumption \eqref{assumption F'}
	\begin{align}
		\EE{\left|\mathsf{M}_N^{-1} -1\right|^{2+\delta} } 
		\leq \EE{\mathsf{M}_N^{-4-2\delta}}^{\frac{1}{2}} \EE{|1-\mathsf{M}_N|^{4+2\delta}}^{\frac{1}{2}}  = O\big( \eps_N^{2+\delta}+\nu_N^{\frac{2+\delta}{2}} \big)=o(|\alpha_N|+\upsilon_N),
	\end{align}
	since $\nu_N=o(1)$. Thus all assumptions of Theorem \ref{Thm convergence perpetuity} are fulfilled.
	
	Regarding case i), the assumption $\nu_N/\eps_N \to 0$ implies $\alpha_N/\upsilon_N \to \infty$ and $\alpha_N/\varepsilon_N \to 1$,	hence $\eps_N \mathsf{Y}_N \to \beta = \frac{\sigma^2}{2}$ in probability by Theorem~\ref{Thm convergence perpetuity}~i).
	
	In case ii), where $\frac{\nu_N}{\eps_N} \to \rho$ with $0<\rho<2$, we have
	\begin{align}
		\frac{\alpha_N}{\upsilon_N }= \frac{\eps_N -\nu_N+ o(\eps_N)}{\nu_N} \to \gamma \text{ with } \gamma:=\frac{1-\rho}{\rho} > - \frac{1}{2}.
	\end{align}
	Thus, by an application of Theorem \ref{Thm convergence perpetuity} ii) $ \nu_N \mathsf{Y}_N$ is asymptotically inverse $\Gamma$-distributed with parameters
	\begin{align}
		(a,b)= (2( 1- \rho )/\rho,\sigma^2).
	\end{align}	
	In case iii), note that $\frac{\nu_N}{\eps_N} \to 2$ entails $\frac{\alpha_N}{\upsilon_N} \to -\frac{1}{2}$. Here we apply Theorem \ref{Thm convergence perpetuity} with $\mathsf{B}_N=\mathsf{F}_N^-$. By a subsequence argument we may assume convergence of $\EE{\mathsf B_N}$ to some constant $0 \leq \beta\leq \infty$. Then Lemma \ref{Lemma bound phi(0)} guarantees $\beta >0$, and the inequality $\mathsf{F}_N^- \leq 4 \mathsf{F}_N^+$ from \eqref{eq:bound Phi} yields $\beta < \infty$. By Theorem \ref{Thm convergence perpetuity} iii) we have $\nu_N\mathsf{Y}^-_N \to \infty$ in probability. This concludes the proof. \end{proof}

In preparation for the subsequent lemma on uniform integrability, we insert a maximal inequality.

\begin{lem}\label{lem: maximal inequality}
	Let $S_k:=\zeta_1+ \cdots+ \zeta_k$, $1\le k \le n$, with iid summands fulfilling $\mathbb E[\zeta_1]=0$ and $\mathbb E[|\zeta_1|^{r}]<\infty$ for some $r\ge 2$. Denote $\nu:= \mathbb E[\zeta_1^2]$. Then,   for any $x>0$
	\begin{align}
		\mathbb P\big(\max_{1\le k \le n} S_k >x\big) \le 2\exp\Big( - \frac{x^2}{cn\nu} \Big)+ c n^{1-r/2} \frac{\mathbb E[|\zeta_1|^{r}]}{\nu^{r/2}},
	\end{align}
	with some $c>0$  depending only on $r$.
\end{lem}
\begin{proof} Without loss we may assume $\nu=1$. 
	We use a bound by Fuk and Nagaev as presented in \cite[Corollary 1.8]{Nagaev1979} and saying that for some $c>0$ depending only on $r$ we have
	\begin{align}
		\label{Nagaev estimate}
		\mathbb P(S_n > x) \le \exp\Big(- \frac {x^2}{cn} \Big) + c nx^{-r} \mathbb E[|\zeta_1|^{r}]
	\end{align}
	for all $x>0$. It is known that similar estimates hold as well for $T_n:=\max _{1\le k \le n}S_k$ instead of $S_n$ (see e.g. \cite{Borovkov1972}). For our purpose the following short argument is sufficient. Note that  Chebychev's inequality implies $\mathbb P(S_k>x) \le 1/2$ for $k\le n$ and $x\ge \sqrt {2n}$. Thus for $x\ge 2\sqrt {2n}$ we have
	\begin{align*}
		\mathbb P(T_n >x, S_n \le x/2) &\le \sum_{k=1}^n \mathbb P(T_{k-1}\le x, S_k>x, S_n-S_k>x/2)\\
		&\le \frac 12 \sum_{k=1}^n \mathbb P(T_{k-1}\le x, S_k>x) = \frac 12 \mathbb P(T_n >x),
	\end{align*}
	consequently $\mathbb P(T_n>x)\le \mathbb P(T_n>x)/2 + \mathbb P(S_n>x/2)$ or 
	\begin{align*}
		\mathbb P(T_n >x) \le 2 \mathbb P(S_n > x/2).
	\end{align*}
	Combining this bound with \eqref{Nagaev estimate} yields the lemma's claim for any $x\ge 2\sqrt{2n}$ by suitably adjusting the constant $c>0$. On the other hand, for $0<x\le 2\sqrt{2n}$ we have 
	\begin{align*}2\exp\Big(- \frac {x^2}{cn} \Big) \ge 2\exp\Big(-\frac 8c\Big)\ge 1
	\end{align*}
	for sufficiently large $c$, and our claim is trivially fulfilled.\end{proof}

%

\begin{lem} \label{Lemma uniform integrability X}
	Under the assumptions of Theorem \ref{Theorem Surv prob BPRE}, if $\nu_N=O(\varepsilon_N)$ as $N\to \infty$, then the sequence $(\frac{1}{\eps_N \mathsf X_N}, N \geq 1)$ is uniformly integrable.
\end{lem}

\begin{proof}
	We allow that $\mathsf X_N$ takes the value $\infty$, then the event $1/(\varepsilon_N\mathsf X_N)=0$ will occur. Essentially, this proof is concerned with bounding  the distribution function of $\mathsf X$. 
	
	Set $\mathbb E[\mathsf M]=1+\varepsilon$ with $\varepsilon >0$ and let $n$  be the natural number satisfying $(n-1)\varepsilon < 1 \le n\varepsilon$. Using \eqref{eq:bound Phi} we have for  $x>0$, $y\ge 1$ 
	\begin{align}
		\mathbb P(\mathsf X \le x) \le \mathbb P \Big(\frac 12 \sum_{k=1}^n \frac{F_k^-}{\mu_{k-1}} \le x\Big)
		\le \mathbb P \Big(\frac 12 \sum_{k=1}^n F_k^- \le xy\Big)+ \mathbb P\big( \max_{1\le k < n} \mu_k >y\big).
	\end{align}
	Denoting $M_k -\mathbb E[M_k]=\zeta_k$ we obtain
	\begin{align}
		\mu_k = \prod_{i=1}^k (1+\varepsilon + \zeta_i) \le \prod_{i=1}^k \exp(\varepsilon + \zeta_k) = \exp \Big(k\varepsilon + \sum_{i=1}^k \zeta_i\Big),
	\end{align}
	hence
	\begin{align}
		\max_{1\le k <n} \mu_k \le \exp\Big( 1 + \max_{1\le k <n} \sum_{i=1}^k \zeta_i \Big), 
	\end{align}
	and with $S_k= \zeta_1+ \cdots + \zeta_k$
	\begin{align}
		\mathbb P(\mathsf X \le x)  \le \mathbb P \Big( \sum_{k=1}^n F_k^- \le 2xy\Big)+ \mathbb P\Big( \max_{1 \le k < n} S_k > \log y - 1\Big).
	\end{align}
	Applying Lemma~\ref{lem: maximal inequality} yields for $y>e$ and $r \geq 2$
	\begin{align}
		\mathbb P(\mathsf X \le x)  \le \mathbb P \Big( \sum_{k=1}^n F_k^- \le 2xy\Big)+ 2\exp\Big( - \frac {(\log y-1)^2}{c(n-1)\nu}\Big) + c (n-1)^{1-r/2} \frac{\mathbb E[|\zeta_1|^{r}]}{\nu^{r/2}}
	\end{align}
	with $\nu = \mathbb E[\zeta_1^2]$. Putting $x= e^{-m}/\varepsilon$, $y=q e^m$ with $q>0$ and taking $ (n-1) < 1/\varepsilon\le n $ into account we arrive for $m>1-\log q$ at
	\begin{align}
		\mathbb P\Big(\mathsf X \le \frac{e^{-m}}{\varepsilon}\Big) \le \mathbb P \Big( \sum_{k=1}^n F_k^- \le  2q n \Big) + 2\exp\Big( - \frac {(m+\log q-1)^2\varepsilon }{c\nu}\Big)+ c \frac{\mathbb E[|\mathsf M- \mathbb E[\mathsf M]|^{r}]}{ (n-1)^{r/2-1}\nu^{r/2}}.
	\end{align}
	Now we exploit this formula for the sequence $(\mathsf X_N)$. Note that $\nu_N^{-r/2}\mathbb E[|\mathsf M_N- \mathbb E[\mathsf M_N]|^{r}]=O(1)$ for some $r>4$ by \eqref{assumption F'}. Thus, the rightmost term is of order $O(n^{1-r/2})=O(\varepsilon_N^{r/2-1})$.  Also, by Lemma \ref{Lemma bound phi(0)} the term $\sum_{k=1}^n F_k^-$ may be stochastically bounded from below by $\theta \text{Bi}_n$ with $\theta>0$ and a binomial random variable  $\text{Bi}_n$  with parameters $n$ and $p_0$. Therefore, for  $q>0$ sufficiently small the probability $\mathbb P \Big( \sum_{k=1}^n F_k^- \le  2q n \Big)$ decreases exponentially fast in $n$, respectively in $\varepsilon_N^{-1}$. Finally, by assumption there is a $c>0$ such that $\varepsilon_N \ge 4c \nu_N$ for all $N$. With these ingredients our estimate  shrinks to
	\begin{align}
		\mathbb P\Big(\mathsf X_N \le \frac{e^{-m}}{\varepsilon_N}\Big) \le c' \varepsilon_N^{r/2-1}+\exp\big(-c'(m+\log q-1)^2\big), \label{eq:bound X_N}
	\end{align}
	for $m>1-\log q$ and for $N$ large enough, with some $c'>0$.
	
	We are  ready to prove the lemma's claim. Note that $\mathsf X_N\ge 1$ a.s. by \eqref{prob and expec combined}. Thus, we have 
	\begin{align}
		\mathbb E\Big[\frac{1}{\eps_N \mathsf X_N} ; \frac{1}{\eps_N \mathsf X_N} \geq e^{m_0}\Big] 
		\leq  \sum_{m=m_0}^{\lfloor \log 1/\eps_N   \rfloor }   e^{m+1}\mathbb P\Big(\frac{e^{-(m+1)}}{\eps_N }\le\mathsf X_N\leq \frac{e^{-m}}{\eps_N }\Big) 
	\end{align}	
	From \eqref{eq:bound X_N} we obtain for $m_0>1-\log q$
	\begin{align}
		\mathbb E\Big[\frac{1}{\eps_N \mathsf X_N} ; \frac{1}{\eps_N \mathsf X_N} \geq e^{m_0}\Big]  
		&\le  \sum_{m=m_0}^{\lfloor \log 1/\eps_N   \rfloor }   e^{m+1} \Big( c' \varepsilon_N^{r/2-1}+ 2e^{-c'(m+\log q-1)^2} \Big)\\
		& \le c'e^2 \varepsilon_N^{r/2-2}+2\sum_{m=m_0}^{\infty }   e^{m+1-c'(m+\log q-1)^2}.
	\end{align}	
	Since the series is convergent and $r>4$, the right-hand expression can be made arbitrarily small with increasing $N$ and $m_0$, which confirms our claim. \end{proof}

\begin{prop} \label{Lemma convergence in prob}
	Under the assumptions of Theorem \ref{Theorem Surv prob BPRE} and  $\nu_N\to 0$ we have as $N \to \infty$,
	\begin{itemize}
		\item[i)] if $\frac{\nu_N}{\varepsilon_N} \to 0$, then $\frac{\pi(\mathsf V_N)}{\varepsilon_N} \to \frac 2{\sigma^2}$ in probability,
		\item[ii)] if $\frac{\nu_N}{\varepsilon_N} \to \rho$ with $0<\rho < 2$, then $\frac{\pi(\mathsf V_N)}{\varepsilon_N}$ is asymptotically $\Gamma$-distributed  with density $c' x^{a'-1}e^{-b'x}$, where $a'= \frac{2-\rho}\rho$, $b'=\frac{\sigma^2}\rho$ and $c'=(b')^{a'}/\Gamma(a')$,
		\item[iii)] if $\frac{\nu_N}{\varepsilon_N} \to 2$, then $\frac{\pi(\mathsf V_N)}{\varepsilon_N}\to 0$  in probability.
	\end{itemize}
\end{prop}

\begin{proof} We begin with some considerations concerning the cases i) and ii). We are going to apply Proposition \ref{Prop expression survival prob}. Note that its assumption $\EE{\log^+ \mathsf{M}_N^{(2)}}<\infty$ is satisfied because of \eqref{assumption MN}. Moreover, $\EE{\log \mathsf{M}_N}>0$ for large $N$ because of Lemma \ref{Lemma Moments of F} ii) and $0 \leq \rho <2$. Also, taking Lemma \ref{Lemma convergence Y} into account it is sufficient to show that $\eps_N |\mathsf Y_N-\mathsf X_N|$ converges to $0$ in probability. Using \eqref{formula X}, \eqref{formula Y} and Lemma \ref{complex approx} we have 
	\begin{align}
		|\mathsf Y-\mathsf X|\leq  \sum_{k=1}^{\infty} \frac{B_k}{\mu_{k-1}},
	\end{align}
	here now with
	\begin{align}
		B_k:=\eta (( M_k^{(2)})^2 +  M_k^{(2)} + \mathbb E[  U_{1,k}^4 \mid \mathsf V]) + 2 F_k^+ \1_{\{\pi_{k,\infty}(\mathsf V)> (\eta  M_k)^3 \text{ or } \eta  M_k >1\}}
	\end{align}	 
	and any $0<\eta \le 1/2$. Thus, we are once more in the setting of Lemma \ref{Lemma tightness} with $ A_k:= 1/ M_k$, $C_k=\mu_k^{-1}$ and
	\begin{align}
		\mathbb E[ \mathsf A^u] &= \mathbb E [\mathsf M^{-u}],  \\
		\mathbb E[\mathsf B]&= \eta(\mathbb E[( \mathsf M^{(2)})^2] +  \mathbb E[\mathsf M^{(2)} ]+ \mathbb E[  \mathsf U^4 ])+ 2\mathbb E[\mathsf  F^+ ;\pi(\mathsf V)> (\eta \mathsf  M)^3 \text{ or } \eta  \mathsf M>1]. \label{eq:expression E[B]}
	\end{align}
	Note that the sequence $A_k, k \geq 1$ is iid and that the $B_k$ have the same finite expectation. 
	Now the conclusion of Lemma \ref{Lemma tightness} reads
	\begin{align}
		\mathbb P\big( (1- \mathbb E[\mathsf A^u]^{\frac 1u})| \mathsf X-\mathsf Y| > c \mathbb E[\mathsf B]\big) \le 2 c^{-\frac u{1+u}} \label{eq:inequality perpetuity}
	\end{align}
	for any $c>0$, provided that $\EE{\mathsf{A}^u }^{\frac{1}{u}}<1$.
	
	In order to evaluate \eqref{eq:inequality perpetuity} we first show that $\EE{\mathsf{B}_N}\to 0$. We have 
	\begin{align}
		\pp{\pi(\mathsf{V}_N) > (\eta \mathsf{M}_N)^3 \text{ or }  \eta \mathsf{M}_N >1 } \to 0,
	\end{align}
	since $\mathsf{M}_N \to 1$ and $\pi(\mathsf{V}_N)\to 0$ in probability, where we use $\pi(\mathsf{V}_N)= \mathsf{X}_N^{-1}$ by Proposition \ref{Prop expression survival prob} and the fact that $\mathsf{X}_N \to \infty$ in probability by Lemma \ref{Lemma uniform integrability X}. Since $\mathsf{F}_N^+$ is uniformly integrable by \eqref{uniform integrability varphi}, the right most term in \eqref{eq:expression E[B]} vanishes as $N\to \infty$.
	Further, by assumption \eqref{assumption F'} we have
	\begin{align}
		\mathbb E[( \mathsf M_N^{(2)})^2] +  \mathbb E[\mathsf M_N^{(2)} ]+ \mathbb E[  \mathsf U_N^4 ] \leq  C,
	\end{align}
	for some constant $C<\infty$ uniformly in  $N$. Altogether, from \eqref{eq:expression E[B]} we obtain $\EE{\mathsf{B_N}} \to 0$, since $\eta$ can be chosen arbitrarily small. Now we are ready to treat the cases i) to iii).
	
	i) From $\nu_N = o(\eps_N)$ by Lemma \ref{Lemma Moments of F} i) we get
	\begin{align}
		1- \EE{\mathsf{A}_N^u}^{\frac{1}{u}} = 1- (1- u \eps_N +  o(\eps_N))^{\frac{1}{u}} = \eps_N + o(\eps_N).
	\end{align}
	Therefore, making use of \eqref{eq:inequality perpetuity} and $\EE{\mathsf B_N}\to 0$ we get $\eps_N|\mathsf{Y}_N-\mathsf{X}_N|\to 0$ in probability as $N \to \infty$. By Lemma \ref{Lemma convergence Y}~i), $\eps_N \mathsf{X}_N \to \frac{\sigma^2}{2}$ in probability, hence our claim follows by Proposition \ref{Prop expression survival prob}. 
	
	ii) Lemma \ref{Lemma Moments of F} i) yields
	\begin{align}
		1-\EE{\mathsf{A}_N^u}^{\frac{1}{u}} = 1- \left(1-u \eps_N + \frac{u(u+1)}{2} \nu_N + o(\eps_N + \nu_N)\right)^{\frac{1}{u}} =\eps_N - \frac{u+1}{2} \nu_N +o(\eps_N + \nu_N). 
	\end{align}
	Here $\frac{\nu_N}{\eps_N} \to \rho \in (0,2)$, hence choosing $u$ small enough such that $\eps_N -\frac{(u+1) \nu_N}{2} \geq  u \nu_N$ for large $N$ and again by \eqref{eq:inequality perpetuity} and $\EE{\mathsf B_N}\to 0$ we get that $\nu_N |\mathsf{Y}_N-\mathsf{X}_N | \to 0$ in probability.
	Consequently, $\nu_N \mathsf{X}_N$ is asymptotically inverse $\Gamma$-distributed by Lemma \ref{Lemma convergence Y}~ii). Hence $\frac{1}{\eps_N \mathsf{X}_N}$ is asymptotically $\Gamma$-distributed with the stated density. Therefore, the claim follows by Proposition \ref{Prop expression survival prob}.
	
	iii) In this case equation \eqref{eq:inequality perpetuity} is no longer applicable. Here we distinguish two cases. By a subsequence argument we may assume that either $\EE{ \log \mathsf{M}_N}\leq 0$ for all $N$, or $\EE{ \log \mathsf{M}_N}>0$ for all $N$. In the first case $\pi_N=0$ by criticality or subcriticality. In the second case we may again apply Proposition \ref{Prop expression survival prob}. Then we have
	\begin{align}
		\frac{\pi(\mathsf{V}_N)}{\eps_N} = \frac{1}{\eps_N \mathsf{X}_N} \leq \frac{1}{2 \eps_N \mathsf{Y}^-_N} \to 0,
	\end{align}
	in probability by \eqref{eq:bound Phi} and by an application of Lemma \ref{Lemma convergence Y}~iii).\end{proof}

\begin{proof}[Proof of Theorem \ref{Theorem Surv prob BPRE}]
	i) Since $\pi(\mathsf{V}_N )/ \eps_N$ is uniformly integrable by Lemma \ref{Lemma uniform integrability X}, it follows
	\begin{align}
		\frac{\pi_N}{\eps_N} = \EE{\frac{\pi(\mathsf{V}_N)}{\eps_N}} \to \frac{2}{\sigma^2},
	\end{align}
	by Proposition \ref{Lemma convergence in prob} i).
	
	ii) Now we have $\nu_N/\eps_N \to \rho \in (0,2)$. By the same line of arguments we obtain
	\begin{align}
		\frac{\pi_N}{\eps_N }=	\EE{\frac{\pi(\mathsf{V}_N)}{\eps_N} } \to \frac{(b')^{a'}}{\Gamma(a')} \int_{0}^{\infty} x x^{a'-1} e^{-b'x} dx = \frac{a'}{b'} =  \frac{2 -\rho }{\sigma^2}.
	\end{align}
	
	iii) This claim follows in the same vein using Lemma \ref{Lemma uniform integrability X} and Proposition \ref{Lemma convergence in prob} iii).
	
	iv) Again, in view of a subsequence argument, we may assume without loss of generality that the limits $\nu_\infty= \lim_N \nu_N$ and $l_\infty = \lim_N \mathbb E[\log \mathsf M_N]$ exist. Because of $\varepsilon_N \to 0$ it follows that
	\begin{align} \mathbb E[f(\mathsf M_N)] \to l_\infty,
	\end{align}
	with $f(x):=\log x- x+1$, $x>0$. Since $f(x)\le 0$ for all $x$, we have $l_\infty\le 0$
	(possibly with value $-\infty$). Note that $f$ is  concave with a single zero at point 1. Thus the condition $l_\infty=0$ entails that $\mathsf M_N$ converges to 1 and $\mathsf M_N- \mathbb E[\mathsf M_N]$ to 0 in probability, as $N\to \infty$. 
	
	Now, if $\nu_\infty=0$, then we may resort to Lemma \ref{Lemma Moments of F} ii), yielding under the present assumption  $\mathbb E[\log \mathsf M_N]<0$ for large $N$, thus subcriticality.
	
	On the other hand, if $\nu_\infty >0$, then  we have uniform integrability of $(\mathsf M_N-\mathbb E[\mathsf M_N])^2$ because of \eqref{assumption F'}, hence
	\begin{align} \liminf_N \mathbb E[ \min(\eta,(\mathsf M_N-\mathbb E[\mathsf M_N])^2) ] \ge \frac {\nu_\infty}{2}>0,
	\end{align}
	for some $\eta >0$. Consequently, $\mathsf M_N- \mathbb E[\mathsf M_N]$ does not converge to 0 in probability. As just shown, this implies $l_\infty<0$, which again implies the claimed subcriticality.\end{proof}

\noindent \textbf{Acknowledgements.}
We like to thank the anonymous referees for their careful reading and helpful suggestions. In particular we are grateful for pointing us to a serious gap in a previous version of the paper.


\bibliographystyle{plain} 
\bibliography{mybib_bernoulli.bib}       



\end{document}